\documentclass[12pt,a4paper]{amsart}
 \usepackage{amssymb,amsthm}
\usepackage{amssymb}
\usepackage{amsthm}
\usepackage{latexsym}
\usepackage{pstricks}
\usepackage{amsbsy}
\usepackage{tikz-cd}

\usepackage{array}
\usepackage[all]{xypic}
\usepackage{amscd}
\usepackage{mathrsfs}
\usepackage{amsfonts}
\usepackage{graphicx}
\usepackage{url}
\usepackage{bookmark}
\usepackage{mathtools}

\usepackage{tikz}
\usepackage{enumerate,multicol}
\usetikzlibrary{arrows}
\usetikzlibrary{positioning}
\xyoption{matrix}
\xyoption{arrow}

\newtheorem{theorem}{Theorem}[section]
\newtheorem*{theorem*}{Theorem}
\newtheorem{lemma}[theorem]{Lemma}
\newtheorem{corollary}[theorem]{Corollary}
\newtheorem{proposition}[theorem]{Proposition}
\newtheorem{definition}[theorem]{Definition}
\theoremstyle{definition}
\newtheorem{remark}[theorem]{Remark}
\newtheorem{example}[theorem]{Example}

\begin{document}
 
 \providecommand{\Gen}{\mathop{\rm Gen}\nolimits}%
\providecommand{\cone}{\mathop{\rm cone}\nolimits}%
\providecommand{\Sub}{\mathop{\rm Sub}\nolimits}%
\providecommand{\rad}{\mathop{\rm rad}\nolimits}%
\providecommand{\coker}{\mathop{\rm coker}\nolimits}%
\def\A{\mathcal{A}}
\def\C{\mathcal{C}}
\def\D{\mathcal{D}}
\def\P{\mathcal{P}}
\def\X{\mathbb{X}}
\def\Y{\mathbb{Y}}
\def\E{\mathcal{E}}
\def\Z{\mathcal{Z}}
\def\K{\mathcal{K}}
\def\F{\mathcal{F}}
\def\T{\mathcal{T}}
\def\U{\mathcal{U}}
\def\V{\mathcal{V}}
\def\W{\mathsf{W}}
\def\PP{{\mathbb P}}
\providecommand{\add}{\mathop{\rm add}\nolimits}%
\providecommand{\End}{\mathop{\rm End}\nolimits}%
\providecommand{\Ext}{\mathop{\rm Ext}\nolimits}%
\providecommand{\Hom}{\mathop{\rm Hom}\nolimits}%
\providecommand{\ind}{\mathop{\rm ind}\nolimits}%
\newcommand{\module}{\mathop{\rm mod}\nolimits}%

\newcommand{\sbt}{\,\begin{picture}(-1,1)(-1,-3)\circle*{3}\end{picture}\ }

\title[Cluster Groups of Mutation-Dynkin Type $A_{n}$]{A Lattice Isomorphism Theorem for Cluster Groups of Mutation-Dynkin Type $A_{n}$}

\author{Isobel Webster}
\address{School of Mathematics \\ 
University of Leeds \\ 
Leeds, LS2 9JT \\ 
United Kingdom \\
}
\email{mm09iw@leeds.ac.uk}
\date{November 2018}
 \maketitle
\begin{abstract}
Each quiver appearing in a seed of a cluster algebra determines a corresponding group, which we call a cluster group, which is defined via a presentation.  Grant and Marsh showed that, for quivers appearing in seeds of cluster algebras of finite type, the associated cluster groups are isomorphic to finite reflection groups. As for finite Coxeter groups, we can consider parabolic subgroups of cluster groups.  We prove that, in the type $A_{n}$ case, there exists an isomorphism between the lattice of subsets of the defining generators of the cluster group and the lattice of its parabolic subgroups. Moreover, each parabolic subgroup has a presentation given by restricting the presentation of the whole group. 
\end{abstract}
\thanks{ }

\maketitle
\section{introduction}

This paper examines cluster groups, which arise from the theory of cluster algebras. Cluster algebras were introduced by Fomin and Zelevinsky in order to construct a model for dual canonical bases in semisimple groups \cite{FZ1}. However, their applications have spread into numerous areas of Mathematics such as combinatorics \cite{FZ2}, quiver representation theory \cite{MRZ} and Poisson geometry \cite{GSV}.  

A (skew-symmetric) cluster algebra \cite[Section 2]{FZ1} is a subalgebra of the field of rational functions, $\mathbb{F} = \mathbb{Q}(u_{1}, u_{2},..., u_{m})$, on a finite number of indeterminates. It is determined by an initial input, known as a \textit{seed}, consisting of a free generating set of $\mathbb{F}$ over $\mathbb{Q}$, i.e. a quiver with no loops or two-cycles. A combinatorial process called \textit{mutation} is repeatedly applied to produce more seeds. A generating set for the cluster algebra is then chosen from the union of the seeds. The elements of this generating set are known as the \textit{cluster variables} of the cluster algebra. This notion of mutating, which involves a local change in the quiver relating to a choice of vertex, is at the heart of the definition of a cluster algebra. 

A cluster algebra is of finite type if it has finitely many cluster variables. The classification \cite{FZ} of the cluster algebras of finite type corresponds to the classification of the finite root systems via Cartan matrices and simply-laced Dynkin diagrams. In this classification it is shown that a cluster algebra is of finite type if and only if it has a seed containing a cluster quiver which is an orientation of a Dynkin diagram (or a disjoint union of Dynkin diagrams). A cluster quiver is said to be of mutation-Dynkin type if it lies in some seed of a cluster algebra of finite type or, equivalently, if there is a finite sequence of mutations taking it to an oriented Dynkin diagram (or a disjoint union of oriented Dynkin diagrams). 

The article \cite{BM} associates to each simply-laced mutation-Dynkin quiver a group presentation and shows that the corresponding group is isomorphic to a finite reflection group. Moreover, this reflection group will be of the same Dynkin type as the cluster algebra from which the quiver arises.

Several subsequent publications have built on this work. The article \cite{FT} provides similar presentations for affine Coxeter groups while \cite{GM} and \cite{HHLP}, independently of one another, extended the results of \cite{BM} to provide presentations for Artin braid groups in the simply-laced case and the finite type case, respectively. Moreover, it is shown that the groups corresponding to both of these presentations are invariant under mutation of the quiver. 

Consider the group presentation associated to each simply-laced mutation-Dynkin quiver in \cite{GM}, together with the additional set of relations that specify that the square of each generator is equal to the identity. Then \cite[Lemma 2.5]{GM} shows the resulting group is isomorphic to the group with presentation defined by \cite{BM}. Thus, we obtain another presentation associated to a mutation-Dynkin quiver that gives a group that is isomorphic to a finite reflection group. It is the group presentation based on the work done in \cite{GM} that is considered here. In particular, we will be considering the group presentations associated to mutation-Dynkin quivers of type $A_{n}$. However, these presentations make sense for any cluster quiver and may be interesting to consider more generally. Due to the context given above, we call the corresponding group a `cluster group' with the aim of exploring its properties in general.

\begin{definition} Let $Q$ be a cluster quiver. Following \cite{FT}, if $i$ and $j$ are vertices of $Q$ which are joined by a single arrow, then we call this arrow \textbf{simple}. 
\end{definition}

\begin{definition}\label{2.7} Let $Q$ be a cluster quiver with vertex set $Q_{0}$. The \textbf{cluster group associated to $Q$}, denoted by $G_{Q}$, is defined to be the group with group presentation $\langle T | R \rangle$ where $$T = \lbrace t_{i} : i \in Q_{0} \rbrace$$ and $R$ is the following set of relations:
\begin{itemize}
\item[$(a)$] For all $i \in Q_{0}$, $t_{i}^{2} = e$.
\item[$(b)$] \textbf{The braid relations:} For all $i, j \in Q_{0}$,
\begin{itemize}
\item[] 
\begin{itemize}
\item[(i)] $t_{i}t_{j} = t_{j}t_{i}$ if there is no arrow between $i$ and $j$ in $Q$.
\item[(ii)] $t_{i}t_{j}t_{i} = t_{j}t_{i}t_{j}$ if $i$ and $j$ are joined by a simple arrow in $Q$. 
\end{itemize} 
\end{itemize} 
\item[$(c)$] \textbf{The cycle relations:}
\begin{align*}
t_{i_{1}}t_{i_{2}}t_{i_{3}}...t_{i_{r}}t_{i_{1}}...t_{i_{r-2}} &= t_{i_{2}}t_{i_{3}}...t_{i_{r}}t_{i_{1}}t_{i_{2}}t_{i_{3}}...t_{i_{r-1}} \\
&= t_{i_{3}}...t_{i_{r}}t_{i_{1}}t_{i_{2}}t_{i_{3}}t_{i_{4}}...t_{i_{r}} \\
&= ... \\
&= t_{i_{r}}t_{i_{1}}t_{i_{2}}t_{i_{3}}...t_{i_{r}}t_{i_{1}}...t_{i_{r-3}}
\end{align*} for every chordless, oriented cycle $$i_{1} \longrightarrow i_{2} \longrightarrow i_{3} \longrightarrow ... \longrightarrow i_{r} \longrightarrow i_{1}$$ in $Q$, in which all arrows are simple. 
\end{itemize} 
\end{definition}

\begin{remark} The relations $(b)(ii)$ and $(c)$ are defined to be compatible with the group presentations associated to diagrams of finite type given in \cite{GM} and the group presentations associated to diagrams of affine type given in \cite{FT}.
\end{remark}

For general cluster quivers the corresponding cluster group is not necessarily invariant under mutation of the quiver \cite[Section 7]{FT}. However, for this paper we will consider only simply-laced mutation-Dynkin quivers. In this case, as we will discuss in Section 2, the cluster group presentation is compatible with quiver mutation \cite[Lemma 2.5]{GM}, \cite[Theorem 5.4]{BM}.

Let $\mathcal{I}$ denote the set of subsets of $T$. Given a cluster quiver $Q$ and $I \in \mathcal{I}$, we use $G_{I}$ to denote the subgroup of the cluster group $G_{Q}$ generated by $I$. We call any subgroup of $G_{Q}$ obtained in this way a parabolic subgroup. We state the first of our two main results here, which will be proved in Section 3.

\begin{theorem*}(Theorem $\ref{latticeiso}$) Suppose $Q$ is a quiver of mutation-Dynkin type with connected components $Q_{1},..., Q_{r}$ of type $A_{n_{1}},..., A_{n_{r}}$ for some $n_{1},..., n_{r} \in \mathbb{Z}^{+}$, then there exists a lattice isomorphism $$\Phi: \mathcal{I}\longrightarrow \tilde{\mathcal{G}},$$
$$\Phi: I\longmapsto G_{I},$$ where $\tilde{\mathcal{G}}$ denotes the collection of all the parabolic subgroups of $G_{Q}$. 
\end{theorem*}

In the case when the quiver is an oriented simply-laced Dynkin diagram, $\Delta$, the associated cluster group presentation is precisely the Coxeter presentation of type $\Delta$. By \cite{BM} and \cite[Lemma 2.5]{GM}, the cluster group presentations associated to mutation-Dynkin quivers give groups which are isomorphic to finite reflection groups and thus are finite Coxeter groups. This can also be shown directly by a simple extension of \cite[Proposition 2.9]{GM}. Under this isomorphism, the defining generators of $G_{Q}$ are mapped to elements of the set of reflections, $\lbrace ws_{i}w^{-1} : w \in W_{\Delta}, s_{i} \in S \rbrace$, of the Coxeter system $(W_{\Delta}, S)$. Thus for any $I \subseteq T$, the parabolic subgroup $G_{I}$ is isomorphic to a reflection subgroup of $W_{\Delta}$. Reflection subgroups of Coxeter groups are themselves Coxeter groups \cite{Dyer}. 

Each finite Coxeter group has an associated root system and conversely each root system defines a Coxeter group (see \cite{Humphreys}). Suppose $\Delta$ is a Dynkin diagram of type $A_{n}$ and $W_{\Delta} $ is the corresponding Coxeter group with root system $\Phi$. For any reflection subgroup $W'$ of $W_{\Delta} $, the subset $$\Psi = \lbrace \alpha : s_{\alpha} \in W' \rbrace \subseteq \Phi $$ is a subsystem of $\Phi$ and $W'$ is the Coxeter group defined by $\Psi$ \cite{DL}. By \cite[Corollary 1]{DL}, the Dynkin diagram of $\Psi$ will be of type $A_{n'}$ for some $n' \leq n$. Thus $W'$ will be of type $A_{n'}$. It follows that, when $Q$ is of mutation-Dynkin type $A_{n}$, each parabolic subgroup of $G_{Q}$ is isomorphic to a cluster group associated to a quiver of type $A_{n'}$, for some $n' \leq n$. In fact, our second main result shows that each parabolic subgroup has a cluster group presentation given by restricting the presentation of the whole group. 

Given any $I \subseteq T$, let $Q_{I}$ denote the full subquiver of $Q$ on the vertices corresponding to the elements of $I$. To distinguish between the defining generators of $G_{Q}$ and $G_{Q_{I}}$, we label the generators of the cluster presentation of $G_{Q_{I}}$ by $t'_{i}$, for each $t_{i} \in I$. 

\begin{theorem*}(Theorem $\ref{paraboliccluster}$) Suppose $Q$ is a quiver of mutation-Dynkin type with connected components $Q_{1},..., Q_{r}$ of type $A_{n_{1}},..., A_{n_{r}}$ for some $n_{1},..., n_{r} \in \mathbb{Z}^{+}$. For any $I \subset T$, there exists an isomorphism 
\begin{align*}
G_{Q_{I}} &\longrightarrow G_{I}, \\
t'_{i} &\longmapsto t_{i}.
\end{align*}
\end{theorem*}

This theorem is our second main result, which will be proved in Section 4. 

Note that for quivers of affine type, the cluster group defined as in Definition $\ref{2.7}$ is not always a Coxeter group. The article \cite{FT} shows that, in order for this to happen, additional relations must be added in some cases. 

There are many well-established results for Coxeter presentations and we are interested in whether cluster groups possess properties comparable  to those of Coxeter groups. Theorem $\ref{latticeiso}$ is analogous to the following theorem for Coxeter groups. 

\begin{theorem}\label{CoxeterLIT}\cite[Theorem $5.5(c)$]{Humphreys} Let $(W, S)$ be a Coxeter system. The assignment $I \mapsto W_{I}$ defines a lattice isomorphism between the collection of subsets of $S$ and the collection of parabolic subgroups $W_{I}$ of $W$.
\end{theorem}

Similarly, analogously to our second main result, any parabolic subgroup of a Coxeter group has a Coxeter presentation \cite[Theorem 5.5(a)]{Humphreys}

In Section 1, we will recall and develop some definitions and theorems from which we will build the required foundations of the theory of cluster groups. In particular, we will highlight results from \cite{BM} and \cite{GM} that show that the cluster group associated to any quiver of mutation-Dynkin type is invariant under mutation of the quiver. For the subsequent sections we consider only quivers of mutation-Dynkin type with connected components $Q_{1},..., Q_{r}$ of type $A_{n_{1}},..., A_{n_{r}}$ for some $n_{1},..., n_{r} \in \mathbb{Z}^{+}$. Let $Q$ be such a quiver. The quivers of mutation-Dynkin type $A_{n}$ arise from triangulations of a disk (see \cite{FWZ}) which provides a geometric insight that we need when considering results for cluster presentations of this type. In Section 2, we use this interpretation to associate a graph to $Q$, called the braid graph. This becomes a tool for constructing an isomorphism between the cluster group $G_{Q}$ and a Young subgroup of $\Sigma_{n+r}$. In Section 3, the existence of such an isomorphism is employed to prove the above lattice isomorphism theorem for cluster groups of mutation-Dynkin type $A_{n_{1}}\sqcup...\sqcup A_{n_{r}}$. Finally, in Section 4, it will be shown that any parabolic subgroup $G_{I}$ of $G_{Q}$ has a cluster group presentation associated to $Q_{I}$.

The author would like to thank the referee for their careful reading of the manuscript and their constructive input.
\section{foundational theory of cluster groups}

First, we recall some definitions and theorems that we will need and then use them too prove results for the theory of cluster groups.

Recall that a quiver is an oriented graph $Q = (Q_{0}, Q_{1})$ where $Q_{0}$ is the set of vertices and $Q_{1}$ is the set of arrows. Moreover, a cluster quiver is one which has no loops or $2$-cycles. Specifically, we consider quivers where the arrows are unlabelled and whose vertices are labelled by the set $\lbrace 1,..., n \rbrace$, where $ n \in \mathbb{N}$ equals the number of vertices.

\begin{definition}\label{2.11} Let $Q$ be a cluster quiver  with vertex set $Q_{0}$. The \textbf{mutation} of $Q$ at a vertex $v \in Q_{0}$ is the quiver $\mu_{v}(Q)$ which is obtained by applying the following steps: 
\begin{itemize}
\item[(a)] For each path in $Q$ of length two passing through the vertex $v$ of the form $u \longrightarrow v \longrightarrow w$, add an arrow $u \longrightarrow w$. Note that this takes into account multiplicity of arrows, i.e. if there are $x$ arrows $u \longrightarrow v$ and $y$ arrows $v \longrightarrow w$ then $xy$ arrows are added from $u$ to $w$.
\item[(b)] Reverse each arrow adjacent to $v$.
\item[(c)] Repeatedly delete pairs of arrows forming 2-cycles until there are no 2-cycles.
\end{itemize}
Note that the operation in $(c)$ is well-defined as, no matter which pairs of arrows we choose to delete, the same number of arrows will be deleted in each direction.
\end{definition}

\begin{definition} Two cluster quivers are \textbf{mutation-equivalent} if there exists a finite sequence of mutations transforming one into the other. 
\end{definition}

\begin{definition} Let $\Delta$ be a simply-laced Dynkin diagram on $n$ vertices which are labelled $1,..., n$. Any orientation of $\Delta$ is called a \textbf{cluster quiver of Dynkin type} $\Delta$.  Following \cite{FWZ}, we call a connected Dynkin diagram \textbf{indecomposable} and a disconnected Dynkin diagram \textbf{decomposable}, where the type is determined by specifying each type of the connected (thus indecomposable) components. The decomposable Dynkin diagram that is the disjoint union of the Dynkin diagrams $\Delta_{1}$ and $\Delta_{2}$ is denoted by $\Delta_{1}\sqcup\Delta_{2}$.
\end{definition}

\begin{definition} A cluster quiver is of \textbf{mutation-Dynkin type $\Delta$} if it is mutation-equivalent to a cluster quiver of Dynkin type $\Delta$.  
\end{definition} We remark that any mutation-Dynkin quiver has a unique type \cite[Theorem 1.7]{FZ}.

Given a cluster quiver $Q$, we consider the cluster group $G_{Q}$, as given in Definition $\ref{2.7}$, where the set of defining generators is denoted by $T$ and the set of relations is denoted by $R$. In particular, suppose that $Q$ is a decomposable cluster quiver with connected components $Q_{1},..., Q_{r}$. We consider the cluster presentation associated to  $Q_{i}$ for each $1 \leq i \leq r$. 

\begin{theorem}\label{directproduct} \cite[Theorem 1]{BTMPR} Let $G_{1}$ and $G_{2}$ be groups with group presentations $\langle X_{1} \vert R_{1} \rangle$ and $\langle X_{2} \vert R_{2} \rangle$, respectively. Then $G_{1} \times G_{2}$ has presentation $\langle X_{1}\cup X_{2} \vert R_{1}\cup R_{2} \cup B \rangle$ where $B$ is the set of relations $ \lbrace x_{1}x_{2} = x_{2}x_{1} : x_{1} \in X_{1}, x_{2} \in X_{2} \rbrace$. 
\end{theorem}
For each $1 \leq i \neq j \leq r$ let $T_{i}$ be the set of generators of $G_{Q_{i}}$ and $R_{i}$ the set of relations and define the set of relations $B_{ij} = \lbrace t_{i}t_{j} = t_{j}t_{i} : t_{i} \in T_{i}, t_{j} \in T_{j} \rbrace$. By Theorem $\ref{directproduct}$, $G_{Q_{1}} \times G_{Q_{2}} \times ... \times G_{Q_{r}} = \langle T_{1}\cup...\cup T_{r} \vert R_{1}\cup ... \cup R_{r} \cup B \rangle$ where $B = \cup_{i=1}^{r}(\cup_{1 \leq j \leq r, j \neq i} B_{ij})$.

However, $T = T_{1} \sqcup ... \sqcup T_{r}$ and as the $Q_{i}$ are the connected components of $Q$, no edge exists between any two vertices lying in different components. Thus $R = R_{1}\cup ... \cup R_{r} \cup B$, meaning $\langle T_{1}\cup...\cup T_{r} \vert R_{1}\cup ... \cup R_{r} \cup B \rangle$ is precisely the cluster presentation of $Q$. This gives the following lemma.

\begin{lemma}\label{directproduct2} If $Q$ is a decomposable cluster quiver with indecomposable components $Q_{1},.., Q_{r}$ then $$G_{Q} = G_{Q_{1}} \times G_{Q_{2}} \times... \times G_{Q_{r}}.$$
\end{lemma} It is easy enough to prove the following lemma. 
\begin{lemma}\label{(1)}
If $Q$ and $P$ are isomorphic cluster quivers, then $G_{Q} \cong G_{P}$.
\end{lemma}

In the light of Lemma \ref{(1)}, we note that $G_{Q}$ only relies on $Q$ up to isomorphism and so nothing is lost by restricting to quivers with vertices labelled by the set $\lbrace 1,..., n \rbrace$. The cluster groups associated to more general cluster quivers will be isomorphic to the groups associated with these quivers. 

\begin{remark} As discussed previously, \cite[Lemma 2.5]{GM} shows the cluster group associated to a mutation-Dynkin quiver is isomorphic to the group with the group presentation defined in \cite{BM}. In fact, \cite[Lemma 2.5]{GM} shows that the cycle relations appearing in the cluster group presentations for mutation-Dynkin quivers are equivalent to those appearing in the group presentations defined in \cite{BM}. Moreover, \cite[Theorem 5.4]{BM} shows that the groups arising from these group presentations are invariant under mutation of the quiver. It follows that the cluster groups associated to mutation-Dynkin quivers are invariant under mutation. Alternatively, by a simple extension of the proof of \cite[Proposition 2.9]{GM} it can be shown directly that when $Q$ is a mutation-Dynkin quiver, the corresponding cluster group will be invariant under mutation of $Q$. Thus, we have the following. 
\end{remark}

\begin{lemma}\label{(2)}\cite[Lemma 2.5]{GM}, \cite[Theorem 5.4]{BM}
Let $Q$ be a cluster quiver of mutation-Dynkin type. If $Q$ and $Q'$ are mutation-equivalent then $G_{Q} \cong G_{Q'}$.
\end{lemma}

By \cite{Humphreys}, each simply-laced Dynkin diagram gives rise to a Weyl group and so to a finite Coxeter group. A finite Coxeter group is given by a simply-laced Dynkin diagram, $\Delta$, in the following way. 

Fix a vertex labelling of $\Delta$ with the set $\lbrace 1,.., n \rbrace$ and consider some orientation, $\overrightarrow{\Delta}$,  of $\Delta$. Let $W_{\Delta}$ be the group with group presentation $<S|R>$ where $S = \lbrace s_{1},..., s_{n} \rbrace$ and $R$ is the set containing the relations $(s_{i}s_{j})^{m(i, j)}$ for all $1 \leq, i, j \leq n$, where $$m(i, j) = \begin{cases} 
      2, & \text{when there is no arrow between } i \text{ and } j \text{ in } \Delta; \\
      3, & \text{when there is an arrow between } i \text{ and } j  \text{ in } \Delta.
   \end{cases}$$ Then $W_{\Delta}$ is a finite Coxeter group \cite[Definition 5.1]{Humphreys}.

Furthermore, as $\Delta$ contains no cycles, $G_{\overrightarrow{\Delta}} = \langle t_{1},..., t_{n} | t_{i}^{2} = (t_{i}t_{j})^{p(i, j)} = e \rangle$ where
\begin{align*}
p(i, j) &= \begin{cases} 
      2, & \text{when there is no arrow between } i \text{ and } j; \\
      3, & \text{when there is an arrow between } i \text{ and } j. 
   \end{cases} \\
   &= m(i, j).
\end{align*}
   
In this case, the cluster group does not depend on the orientation of $\Delta$, so we denote it simply by $G_{\Delta}$. 
   
Thus $$\varphi: G_{\Delta} \longrightarrow W_{\Delta},$$ $$\varphi(t_{i}) := s_{i},$$ defines an isomorphism between $G_{\Delta}$ and $ W_{\Delta}$.

If $Q$ is of mutation-Dynkin type $\Delta$, then $Q$ is mutation-equivalent to some orientation of $\Delta$ (with some fixed labelling of vertices by $1,...,n$). By Lemma~\ref{(2)}, $G_{Q} \cong G_{\Delta} \cong W_{\Delta}$. That is, when $Q$ is a quiver of mutation-Dynkin type, the associated cluster group is isomorphic to the finite Coxeter group of the same type. 

For more general cluster quivers, we consider the following subgroups of the associated cluster group.

\begin{definition} For a subset $I$ of $T$, let $G_{I}$ denote the subgroup of $G_{Q}$ generated by the elements of $I$. A \textbf{parabolic subgroup} of $G_{Q}$ is a subgroup of the form $G_{I}$ for some $I \subseteq T$.
\end{definition}

When $Q$ is a quiver of mutation-Dynkin type $A_{n}$, we will show that there exists a lattice isomorphism between the lattice of subsets of the set of generators of $G_{Q}$ and the lattice of its parabolic subgroups. First, we need to define the \textit{braid graph} of $Q$, which we do in the following section.

\section{the braid graph of a cluster group of mutation-dynkin type $A_{n}$}

Fix a quiver $Q$ of mutation-Dynkin type with connected components $Q_{1},..., Q_{r}$ of types $A_{n_{1}},..., A_{n_{r}}$, respectively, where $n_{1},..., n_{r} \in \mathbb{Z}^{+}$ and $ n = \Sigma_{i=1}^{r} n_{i}$. As previously mentioned, the quivers of mutation-Dynkin type $A_{n}$ arise from triangulations of a polygon \cite{FWZ}. In this section we explain how in the case when $r=1$, we can construct a \textit{braid graph} from a triangulation that gives rise to $Q$. This braid graph will be independent of the choice of triangulation giving rise to $Q$ so we can consider it to be the braid graph of $Q$. This graph is then used to build an isomorphism between $G_{Q}$ and $\Sigma_{n+1}$. This theory will then be extended to the case when $r \geq 1$. 

In a convex polygon, $P$, a \textbf{diagonal} is a line in the interior of $P$ which connects two non-adjacent vertices and only touches the boundary of the polygon at its endpoints. A \textbf{triangulation} of $P$ is a decomposition of the polygon into triangles by a maximal set of non-crossing diagonals. We remark that every triangulation of an $n$-gon has exactly $n-2$ triangles and $n-3$ diagonals. 
 
Every triangulation $\mathcal{T}$ of a convex $(n+3)$-gon gives rise to an indecomposable quiver of mutation-Dynkin type $A_{n}$, denoted by $Q_{\mathcal{T}}$ \cite{CCS}.

\begin{definition}\label{quivertriangulation}\cite[Section 4]{FST} For a triangulation $\mathcal{T}$ of a convex $(n+3)$-gon, $Q_{\mathcal{T}}$ is the quiver whose vertices are in bijection with the diagonals of $\mathcal{T}$. Moreover, there exists an arrow from the vertex $i$ to the vertex $j$ in $Q_{\mathcal{T}}$ if and only if the corresponding diagonals $d_{i}$ and $d_{j}$ bound a common triangle where $d_{j}$ immediately precedes $d_{i}$ in the anticlockwise orientation of the triangle.  
\end{definition} 

\begin{theorem}\label{triangulationquivers} \cite[Lemma 2.1]{CCS} (and also from \cite[Example 6.6]{FST}) For a triangulation $\mathcal{T}$ of a convex $(n+3)$-gon, $Q_{\mathcal{T}}$ is of mutation-Dynkin type $A_{n}$. Conversely, every quiver $Q$ of mutation-Dynkin type $A_{n}$ is of the form $Q_{\mathcal{T}}$ for some triangulation, $\mathcal{T}$, of an $(n+3)$-gon. 
\end{theorem}

A flip along the diagonal $d_{i}$ is defined in the following way. Suppose  $d_{i} = XY$ where $X$ and $Y$ are distinct vertices of $P$. We have that $d_{i}$ is adjacent to two triangles, $XYA$ and $XYB$ where $A$ and $B$ are vertices of $P$, distinct from $X$ and $Y$. A flip along the diagonal $d_{i}$ consists of deleting $d_{i}$ and adding the new diagonal $d'_{i} = AB$ to form a new triangulation  \cite[Definition 3.5]{FST}.

As mentioned in \cite[Proposition 3.8]{FST}, the articles \cite{25}, \cite{26} and  \cite{31} give the following proposition. 
\begin{proposition}\cite{25}, \cite{26}, \cite{31} Any two triangulations of a polygon are connected by a sequence of flips along diagonals.
\end{proposition}
\begin{proposition}\cite[Lemma 2.1]{CCS} Let $\mathcal{T}$ be a triangulation of a convex polygon and let $\mathcal{T'}$ be the triangulation obtained by flipping $\mathcal{T}$ along the diagonal $d_{i}$. Then $Q_{\mathcal{T'}} = \mu_{i}(Q_{\mathcal{T}})$. 
\end{proposition} 

We outline here a neat summary of the above which can be found in \cite{HAT}. Let $\mathbb{T}$ be the set of all triangulations of the $(n+3)$-gon $P$ and $\mathcal{M}_{n}$ be the set of quivers of mutation-Dynkin type $A_{n}$. Define a function $$\gamma: \mathbb{T} \longrightarrow \mathcal{M}_{n},$$ $$\gamma: \mathcal{T} \longmapsto Q_{\mathcal{T}}.$$

This map is surjective as every indecomposable quiver $Q$ of mutation-Dynkin type $A_{n}$ arises from a triangulation of an $(n+3)$-gon. 

Define the following equivalence relation on $\mathbb{T}$:  \begin{center}
$\mathcal{T} \sim \mathcal{T'}$ if and only if $\mathcal{T'}$ can be obtained from $\mathcal{T}$ by a clockwise rotation of $P$.
\end{center} The map $\gamma$ induces a surjective map $\tilde{\gamma}: (\mathbb{T}, \sim) \longrightarrow \mathcal{M}_{n}$. 

\begin{theorem}\label{braidgraphunchanged} \cite[Theorem 3.5]{HAT} For $n \geq 2$, the map $\tilde{\gamma}: (\mathbb{T}, \sim) \longrightarrow \mathcal{M}_{n}$ is bijective.
\end{theorem}

Suppose now that $Q$ is a quiver of mutation-Dynkin type with connected components $Q_{1},..., Q_{r}$ of types $A_{n_{1}},..., A_{n_{r}}$, respectively, where $n_{1},..., n_{r} \in \mathbb{Z}^{+}$ and $ n = \Sigma_{i=1}^{r} n_{i}$ and $r \geq 1$. From the above, it follows that $Q$ arises from a triangulation of the disjoint union of $P_{1},..., P_{r}$ where $P_{i}$ is a convex $(n_{i}+3)$-gon for each $1 \leq i \leq r$. Moreover, any triangulation of this disjoint union of polygons admits a quiver of mutation-Dynkin type $A_{n_{1}}\sqcup...\sqcup A_{n_{r}}$.

\begin{definition}\cite[Definition 3.1]{GM} Let $\mathcal{T}$ be a triangulation of a convex $(n+3)$-gon, $P$. The \textbf{braid graph of $\mathcal{T}$} is the graph $\Gamma_{\mathcal{T}} = (V_{\mathcal{T}}, E_{\mathcal{T}})$ where $V_{\mathcal{T}}$ are the vertices of $V$ and $E_{\mathcal{T}}$ are the edges, defined in the following way. The vertices $V_{\mathcal{T}}$ are in bijection with the triangles of $\mathcal{T}$ in $P$ and there exists an edge between two vertices if and only if the corresponding triangles share a common diagonal in $\mathcal{T}$. 
\end{definition}

For each $\mathcal{T}, \mathcal{T'} \in \mathbb{T}$ with $\mathcal{T} \sim \mathcal{T'}$, it is clear that $\Gamma_{\mathcal{T}}$ is isomorphic to $\Gamma_{\mathcal{T'}}$. 

For the moment, we restrict to the case when $r=1$. Let $\mathcal{T}$ be a triangulation giving rise to $Q$. So $\Gamma_{\mathcal{T}}$ denotes the braid graph associated to $\mathcal{T}$. In fact, all triangulations giving rise to $Q$ will have the same braid graph. Thus it makes sense to refer to $\Gamma_{\mathcal{T}}$ as the braid graph of $Q$ and we will denote this graph by $\Gamma_{Q}$.  

Moreover, $\Gamma_{Q}$ is a connected tree on $n+1$ vertices and, as each triangle can be bounded by between 1 and 3 diagonals, the valancy of each vertex is equal to 1, 2 or 3. 

\begin{example} Let $Q$ be the quiver 
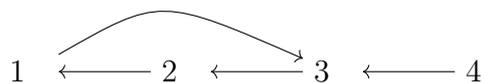
\begin{center}
\begin{tikzpicture}
\node (v0) at (0:0)[text width=0.5cm] {$1$};
\node (v1) at (0:2)[text width=0.5cm] {$2$};
\node (v6) at (30:2) {};
\node (v4) at (0:4)[text width=0.5cm] {$3$};
\node (v5) at (0:6)[text width=0.5cm] {$4$};

\draw[<-] (v0) -- (v1);
\draw[->] (v4) -- (v1);
\draw[<-] (v4) .. controls (v6) .. (v0);
\draw[<-] (v4) -- (v5);
\end{tikzpicture}
\end{center} Consider the following triangulation, $\mathcal{T}$:
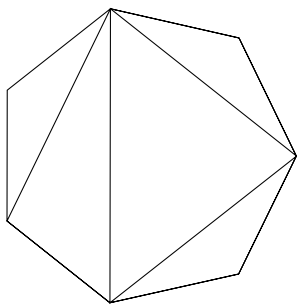
\begin{center}
\begin{tikzpicture} \filldraw (-3,0) node[anchor=east]  {};
   \newdimen\R
   \R=2cm
   \draw (0:\R)
      \foreach \x in {51.4,102.8,154.2,205.6, 257, 308.4, 359.8, 360} {  -- (\x:\R) }
     \foreach \x in {102.8, 257, 359.8, 360} {  -- (\x:\R) }
     \foreach \x in {51.4, 102.8, 205.6, 257, 308.4, 359.8, 360} {  --  (\x:\R)};
\end{tikzpicture}
\end{center} From Definition $\ref{quivertriangulation}$, we have that  $Q_{\mathcal{T}} = Q$. That is,  $\mathcal{T}$ is a triangulation giving rise to $Q$. From this triangulation, we obtain the following braid graph of $Q$. 

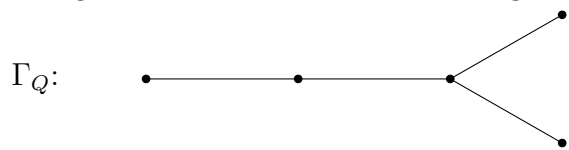
\begin{center}
  \begin{tikzpicture}[scale=1]
    \filldraw (-1,0) node[anchor=east]  {$\Gamma_{Q}$:};
    \foreach \x in {0,1,2}
    \filldraw[black][xshift=\x cm] (\x cm,0) circle (.05cm);
    \filldraw[black][xshift= 4cm] (30: 17 mm) circle (.05cm);
    \filldraw[black][xshift= 4cm] (-30: 17 mm) circle (.05cm);
    \foreach \y in {0, 1}
    \draw[xshift=\y cm] (\y cm,0) -- +(2 cm,0);
    \draw[xshift=4 cm] (30: 0) -- (30: 17 mm);
    \draw[xshift=4 cm] (-30: 0) -- (-30: 17 mm);
  \end{tikzpicture}
\end{center}
\end{example}
Let $\Gamma$ be a planar tree. We denote by $\Sigma_{\Gamma}$ the symmetric group on the vertices of $\Gamma$. That is, the group of permutations of the set of vertices of $\Gamma$. 
\begin{proposition}\label{sympresentation}\cite[Proposition 3.4]{Ser} For a planar tree $\Gamma$, the group $\Sigma_{\Gamma}$ is generated by the set $X_{\Gamma} = \lbrace \sigma : \sigma \text{ is an edge of } \Gamma \rbrace$ subject to the relations:
\begin{itemize}
\item[(1)] $\sigma^{2} = e$ for all $\sigma \in X_{\Gamma}$. 
\item[(2)] If $\sigma_{1}, \sigma_{2} \in X_{\Gamma}$ are disjoint then $$\sigma_{1}\sigma_{2} = \sigma_{2}\sigma_{1}. $$
\item[(3)] If $\sigma_{1}, \sigma_{2} \in X_{\Gamma}$ have one common vertex then $$\sigma_{1}\sigma_{2}\sigma_{1} = \sigma_{2}\sigma_{1}\sigma_{2}. $$
\item[(4)]  If $\sigma_{1}, \sigma_{2}, \sigma_{3} \in X_{\Gamma}$ have a single vertex in common and lie in clockwise order $$\sigma_{1}\sigma_{2}\sigma_{3}\sigma_{1} = \sigma_{2}\sigma_{3}\sigma_{1}\sigma_{2} = \sigma_{3}\sigma_{1}\sigma_{2}\sigma_{3}.$$
\end{itemize} 
\end{proposition}

An additional relation is given in \cite{Ser} for when $\Gamma$ is not a tree. However, as the braid graphs for the quivers we are considering will always be trees, we only require the above four relations. 

\begin{remark} Given the relations $(1)-(3)$, \cite[Lemma 2.5]{GM} shows the relations given in $(4)$ are equivalent to the cycle relation given in the group presentations defined in \cite{BM}. 
\end{remark}

In \cite{Ser}, a version of Proposition $\ref{sympresentation}$ is given for the braid group, in which an edge $\sigma$ is interpreted as a braid which twists the strands corresponding to the endpoints of the edge. In the context of the symmetric group, we view $\sigma$ as the transposition which interchanges the endpoints of the corresponding edge. As noted in \cite[Remark 3.4]{Ser}, the proof that the statement is true for the symmetric group is similar to the proof given for \cite[Proposition 3.4]{Ser} for the braid group.

Choose any labelling of the vertices of $\Gamma_{Q}$ by the set $\lbrace 1, 2, ..., n+1 \rbrace$. There is a bijection between the vertex set of $Q$ (which is in bijection with the diagonals of $\mathcal{T}$) and the edges of $\Gamma_{Q}$ \cite[Definition 3.1]{GM}. So we have a bijection between the vertex set of $Q$ and the set of edges of $\Gamma_{Q}$. We label the edge corresponding to the vertex $i$ under this bijection by $E_{i}$. 

\begin{lemma}\label{braidgraphiso1} Suppose $E_{i}$ has endpoints $x_{i}, y_{i} \in \lbrace 1, 2, ..., n+1 \rbrace$, $x_{i} \neq y_{i}$. Then there exists an isomorphism \begin{align*}
&\pi_{Q}: G_{Q} \longrightarrow \Sigma_{n+1}, \\
&\pi_{Q}: t_{i} \longmapsto (x_{i}, y_{i}).
\end{align*}
\end{lemma}
\begin{proof} By labelling the vertices of the braid graph $\Gamma_{Q}$ of $Q$ by the set $\lbrace 1,..., n+1 \rbrace$, we obtain a presentation from Proposition $\ref{sympresentation}$ for $\Sigma_{n+1}$ with generating set $\lbrace (x_{i}, y_{i}) : \lbrace x_{i}, y_{i} \rbrace \text{ is an edge in } \Gamma_{Q} \rbrace$ subject to the relations $(1) - (4)$. As these correspond to the relations in the cluster presentation of $G_{Q}$ via the bijection between the vertices of $Q$ (and so the set of defining generators of $G_{Q}$) and the edges of the braid graph, it is clear that $\pi_{Q}$ is an isomorphism. 
\end{proof}
\begin{definition}\cite{GordanKerber} For some $n, k \in \mathbb{N}$, let $\rho = \lbrace \alpha_{j} : 1 \leq j \leq k \rbrace$ be a set partition of $\lbrace 1, ..., n \rbrace$. The subgroup of $\Sigma_{n}$ given by $$\Sigma_{\alpha_{1}} \times \Sigma_{\alpha_{2}} \times ... \times \Sigma_{\alpha_{k}}$$ where $\Sigma_{\alpha_{j}} = \lbrace \sigma \in \Sigma_{n} : \sigma(m) = m,  \forall  m \notin \alpha_{j} \rbrace$ is called the \textbf{Young subgroup corresponding to $\rho$} and is denoted by $Y(\rho)$. 
\end{definition}

Next, we consider the case when $r \geq 1$. That is, when $Q$ is a quiver of mutation-Dynkin type with connected components $Q_{1},..., Q_{r}$ of types $A_{n_{1}},..., A_{n_{r}}$, respectively, where $n_{1},..., n_{r} \in \mathbb{Z}^{+}$ and $ n = \Sigma_{i=1}^{r} n_{i}$ and $r \geq 1$.

Let $P$ be the disjoint union of $P_{1},..., P_{r}$, where each $P_{i}$ is a convex $(n_{i}+3)$-gon for each $1 \leq i \leq r$. As previously discussed, there exists a triangulation $\mathcal{T}_{i}$ of $P_{i}$ which gives rise to $Q_{i}$ for each $1 \leq i \leq r$. Let $\mathcal{T}$ be the collection of these triangulations. Thus $\mathcal{T}$ will be a triangulation of $P$ giving rise to $Q$. Again, as the braid graph is independent of the choice of triangulation giving rise to $Q$, we can denote it by $\Gamma_{Q}$. It follows that $\Gamma_{Q}$ will be the graph that is the disjoint union of $\Gamma_{Q_{1}},....,\Gamma_{Q_{r}}$. Note that $\Gamma_{Q}$ will contain $n$ edges and $n + r$ vertices, each with valancy $1$, $2$ or $3$. 

Choose any labelling of the vertices of $\Gamma_{Q_{i}}$ for each $1 \leq i \leq r$ by the set $N_{i} = \lbrace (\Sigma_{j=1}^{i-1} n_{j}) + i,..., (\Sigma_{j=1}^{i}n_{j}) + i \rbrace$, taking $N_{1} = \lbrace 1, 2,..., n_{1}+1 \rbrace$.

The vertex sets of the connected components give a partition $\rho := \sqcup_{j=1}^{r}N_{j}$ of the set $\lbrace 1,..., n+r \rbrace$ and we consider the Young subgroup, $Y(\rho)$, corresponding to $\rho$. 

\begin{lemma}\label{braidgraphiso2} Suppose $E_{i}$ is the edge in $\Gamma_{Q}$ corresponding to the vertex $i$ of $Q$ with endpoints $x_{i}, y_{i} \in \lbrace 1, 2, ..., n+r \rbrace$. Then there exists an isomorphism \begin{align*}
&\pi_{Q}: G_{Q} \longrightarrow Y(\rho), \\
&\pi_{Q}: t_{i} \longmapsto (x_{i}, y_{i}).
\end{align*}
\end{lemma}
\begin{proof} By Lemma $\ref{directproduct2}$, $$G_{Q} = G_{Q_{1}} \times...\times G_{Q_{r}}.$$ By Lemma \ref{braidgraphiso1}, for each $1 \leq j \leq r$ and any labelling of $\Gamma_{Q_{j}}$ by $\lbrace 1,..., n_{j}+1 \rbrace$ there exists an isomorphism $$\pi_{Q_{j}}: G_{Q_{j}} \longrightarrow \Sigma_{n_{j}+1},$$ $$\pi_{Q_{j}}: t_{i} \longmapsto (x'_{i}, y'_{i}),$$ where $E_{i} = \lbrace x'_{i}, y'_{i} \rbrace$ for distinct $x'_{i}, y'_{i} \in \lbrace 1,..., n_{i}+1 \rbrace$. Let $p_{j}: \Sigma_{n_{j}+1} \longrightarrow \Sigma_{N_{j}}$ be a relabelling of $\Gamma_{Q_{j}}$ to the induced labelling by $\Gamma_{Q}$. So the following map $$ p_{j} \circ \pi_{Q_{j}}: G_{Q_{j}} \longrightarrow \Sigma_{N_{j}},$$ $$p_{j} \circ \pi_{Q_{j}}: t_{i} \longmapsto (x_{i}, y_{i})$$ is an isomorphism where $x_{i}, y_{i} \in \lbrace 1, 2, ..., n+r \rbrace$ are the endpoints of $E_{i}$ in the fixed labelling of $\Gamma_{Q}$. 

So we can define an isomorphism $$\pi_{Q}: G_{Q_{1}} \times ... \times G_{Q_{r}} \longrightarrow \Sigma_{N_{1}} \times ... \times \Sigma_{N_{r}}$$ by $\pi \vert_{G_{Q_{j}}} = p_{j} \circ \pi_{Q_{j}}$ for each $1 \leq j \leq r$. 

Noting that $G_{Q} = G_{Q_{1}} \times ... \times G_{Q_{r}}$ and $Y(\rho) = \Sigma_{N_{1}} \times ... \times \Sigma_{N_{r}}$, we have obtained the desired isomorphism. 
\end{proof}
\begin{remark} It follows from Lemma $\ref{braidgraphiso2}$ that the map between the set of vertices, $V$, of $Q$ to the set of defining generators of $G_{Q}$, taking $i\in V$ to $t_{i}$, is injective. That is, for any vertices $i$ and $j$ of $Q$, $t_{i} = t_{j}$ if and only if $i = j$. 
\end{remark}
\section{main theorem: a lattice isomorphism}\label{}

In \cite{Humphreys}, the proof of Theorem $\ref{CoxeterLIT}$ examines how an element of the finite Coxeter group acts on the associated root system. We approach the problem of proving an analogous lattice isomorphism theorem for cluster groups of mutation-Dynkin type $A_{n_{1}}\sqcup...\sqcup A_{n_{r}}$ in a different way. 

First, we recall some key definitions and examples of lattices and partitions of sets that we will need. 

\begin{definition}\cite[Definition 2.4]{Priestley} A \textbf{lattice} is a partially ordered set in which every two element subset has both a least upper bound (the `join') and a greatest lower bound (the `meet'). For any two elements $X$ and $Y$ of a lattice, we denote the \textbf{join} of $X$ and $Y$ by $X\vee Y$ and the \textbf{meet} of $X$ and $Y$ by $X\wedge Y$.
\end{definition}

We give three examples of lattices which will be useful.
\begin{example}\label{latticeeg}
\begin{itemize}

\item[(1.)] Let $X$ be a set. The power set of $X$ forms a partially ordered set under inclusion which is a lattice. 

\item[(2.)] Recall that a \textbf{(set) partition} $\rho$ of a non-empty set $X$ is a collection of non-empty subsets of $X$ such that 
\begin{itemize}
\item[(a)] $X = \bigcup\rho$
\item[(b)] $\alpha_{1}\cap \alpha_{2} = \emptyset$ for every distinct $\alpha_{1}, \alpha_{2} \in \rho$.
\end{itemize} We call the elements of a partition $\rho$ the \textbf{parts} of $\rho$. When $\rho = \lbrace \alpha_{1},..., \alpha_{k}\rbrace$ is a partition of a set $X$ we employ an abuse of notation by writing $\rho = \sqcup_{j=1}^{k}\alpha_{j}$.

The set of partitions of a set $X$ forms a partially ordered set under the refinement ordering. That is, for each pair of partitions $\rho, \rho'$ of $X$, 
$$\rho \leq \rho' \Leftrightarrow \text{ every part of } \rho \text{ is a subset of some part of } \rho'.$$ 

Moreover, \cite[Theorem 5.15.1]{Penner} outlines how the meet and join of any two partitions of the same set are obtained, which we describe below. 

The partitions of $X$ are in bijection with the equivalence relations on $X$. So for each partition $\rho = \sqcup_{j=1}^{k}\alpha_{j}$ of $X$ we have a corresponding equivalence relation, $R_{\rho}$, on $X$ where for each $x, y \in X$, $$ xR_{\rho}y \Leftrightarrow x, y \in \alpha_{j} \text{ for some } 1 \leq j \leq k.$$ 

From any two equivalence relations $R_{1}$ and $R_{2}$ on a set $X$, we construct new equivalence classes, denoted by $R_{1}\cap R_{2}$  and $t(R_{1}\cup R_{2})$, in the following way. For any $x, y \in X$, \begin{itemize}
\item[(a)] $x(R_{1}\cap R_{2})y$ if and only if $ xR_{1}y \text{ and } xR_{2}y$
\item[(b)] $x(t(R_{1}\cup R_{2}))y$ if and only if there exist $ z_{0},..., z_{m} \in X$ such that $ x=z_{0}, y=z_{m}$ and either $ z_{i}R_{1}z_{i+1} $ or $ z_{i}R_{2}z_{i+1}$ for all $ 1\leq i \leq m-1$.
\end{itemize} Note that $t(R_{1}\cup R_{2})$ is the transitive closure of the binary relation $R_{1}\cup R_{2}$ on $X$. 

Given two partitions $\rho_{1}, \rho_{2}$ of the set $X$, $\rho_{1}\wedge \rho_{2}$ is the partition corresponding to $R_{1}\cap R_{2}$ and $\rho_{1}\vee \rho_{2}$ is the partition corresponding to $t(R_{1}\cup R_{2})$ \cite[Theorem 5.15.1]{Penner}.

\item[(3.)] Let $G$ be a group. Then the set of subgroups of $G$ forms a partially ordered set under inclusion. This forms a lattice and, for any subgroups $G_{1}, G_{2}$ of $G$, $G_{1}\vee G_{2} = \langle G_{1}\cup G_{2} \rangle$ and $G_{1}\wedge G_{2}= G_{1}\cap G_{2}$. 
\end{itemize}
\end{example}
\begin{definition}\cite[Definition 2.13]{Priestley} Let $L$ be a lattice. A non-empty subset $K$ of $L$ is a \textbf{sublattice} of $L$ if for every $x, y \in K$, $x\vee y, x \wedge y \in K$. 
\end{definition}
\begin{definition}\cite[Definition 2.16]{Priestley} Let $L$ and $K$ be lattices. A map $\phi: L \longrightarrow K$ is a \textbf{lattice homomorphism} if for all $x, y \in L$, $\phi(x\vee y) = \phi(x)\vee \phi(y)$ and $\phi(x\wedge y) = \phi(x)\wedge \phi(y)$. Moreover, $\phi$ is a \textbf{lattice isomorphism} if it is a bijective lattice homomorphism. 
\end{definition}

\begin{proposition}\label{orderpreserving}\cite[Proposition 2.4(ii)]{Priestley} Let $L$ and $K$ be lattices. For any map $\phi: L \longrightarrow K$, $\phi$ is a lattice isomorphism if and only if $\phi$ is an order-isomorphism. 
\end{proposition}

Again, suppose $Q$ is a quiver of mutation-Dynkin type with connected components $Q_{1},..., Q_{r}$ of types $A_{n_{1}},..., A_{n_{r}}$, respectively, where $n_{1},..., n_{r} \in \mathbb{Z}^{+}$ and $ n = \Sigma_{i=1}^{r} n_{i}$. Take $\mathcal{T}$ to be a triangulation of the disjoint union of $(n_{i}+3)$-gons giving rise to $Q$ and consider the braid graph, $\Gamma_{Q}$, of $Q$. Fix a labelling of $\Gamma_{Q}$ by the set $\lbrace 1, ..., n+r \rbrace$. The previous section outlined how there is a bijection from the set of generators of $G_{Q}$ onto the set of edges of $\Gamma_{Q}$. We let $E_{i}$ represent the edge in $\Gamma_{Q}$ corresponding to the vertex $i$ of $Q$ for each $1 \leq i \leq n$. Moreover, from the chosen labelling of the braid graph, we obtain a partition $\rho$ of $\lbrace 1, ..., n+r \rbrace$ by taking the parts of $\rho$ to be the vertex sets of the connected components of $\Gamma_{Q}$. Suppose $E_{i}$ has endpoints $x_{i}, y_{i} \in \lbrace 1, 2, ..., n+r \rbrace$ ($x_{i}\neq y_{i}$). Then, by Lemma $\ref{braidgraphiso2}$, there exists an isomorphism \begin{align*}
&\pi_{Q}: G_{Q} \longrightarrow Y(\rho), \\
&\pi_{Q}: t_{i} \longmapsto (x_{i}, y_{i}).
\end{align*}

Let $\mathcal{I}$ be the power set of $T = \lbrace t_{1},..., t_{n}\rbrace$. This is a lattice under inclusion. 

Moreover, let $\mathcal{G}$ be the set of subgroups of $G_{Q}$ (so $\mathcal{G}$ is a lattice under inclusion) and let $\tilde{\mathcal{G}} = \lbrace G_{I} : I \in \mathcal{I} \rbrace$, i.e. $\tilde{\mathcal{G}}$ is the collection of parabolic subgroups of $G_{Q}$. Let $\mathcal{P}$ be the set of partitions of $\lbrace 1,..., n+r \rbrace$ and $\mathcal{Y}$ the set of Young subgroups of $\Sigma_{n+r}$. By Example $\ref{latticeeg}$, $\mathcal{P}$ is a  lattice under the refinement ordering. The following result is well-known (see e.g. \cite{BG}). 

\begin{proposition}\label{psi}\cite{BG} The set $\mathcal{Y}$ is a lattice under inclusion and there exists a lattice isomorphism 
\begin{align*}
&\psi: \mathcal{P}\longrightarrow \mathcal{Y}, \\ &\psi: \rho \longmapsto Y(\rho).
\end{align*}
\end{proposition}  

\begin{definition} Given $I \in \mathcal{I}$ we obtain a partition $\rho_{I} \in \mathcal{P}$ in the following way. Let $\Gamma_{I}$ be the graph obtained from $\Gamma_{Q}$ by deleting all edges $E_{i}$ such that $t_{i} \notin I$. As $\Gamma_{Q}$ is a connected tree, $\Gamma_{I}$ will consist of some connected components, $\Gamma^{1}_{I}, \Gamma^{2}_{I}, ..., \Gamma^{k}_{I}$, each of which is a full subgraph of $\Gamma_{Q}$. We define $\rho_{I} = \sqcup_{j=1}^{k}\alpha_{j}$ where $\alpha_{j}$ is the vertex set of the connected component $\Gamma^{j}_{I}$. 
\end{definition} Let $\tilde{\mathcal{P}} = \lbrace \rho_{I} : I \in \mathcal{I}\rbrace $ and $\tilde{\mathcal{Y}} = \lbrace Y(\rho)  : \rho \in \tilde{\mathcal{P}} \rbrace $. In this section, we will show that there exists a lattice isomorphism:
\begin{align*}
&\phi: \mathcal{I}\overset{\text{Lem 4.13}}{\longrightarrow} \tilde{\mathcal{Y}}, \\ &\phi: I\longmapsto Y(\rho_{I}). 
\end{align*}
To do this, we will show there exists an order isomorphism:
\begin{align*}
&\phi_{1}: \mathcal{I}\overset{\text{Lem 4.8}}{\longrightarrow} \tilde{\mathcal{P}}, \\
&\phi_{1}: I\longmapsto \rho_{I}
\end{align*} 
and prove that $\tilde{\mathcal{P}}$ is a sublattice of $\mathcal{P}$ under the refinement ordering. Thus, by Proposition $\ref{orderpreserving}$, $\phi_{1}$ will be a lattice isomorphism. We will further show that $\tilde{\mathcal{Y}}$ is a sublattice of $\mathcal{Y}$ under inclusion and use Proposition $\ref{psi}$ to show the following map is a lattice isomorphism.
\begin{align*}
&\phi_{2}: \tilde{\mathcal{P}}\overset{\text{Lem 4.12}}{\longrightarrow} \tilde{\mathcal{Y}}, \\
&\phi_{2}: \rho_{I} \longmapsto Y(\rho_{I}).
\end{align*}
By composing $\phi_{1}$ and $\phi_{2}$, we obtain the desired lattice isomorphism between $\mathcal{I}$ and $\tilde{\mathcal{Y}}$. This approach is summarised in the following diagram. 

\begin{center}

\begin{tikzcd}
    \mathcal{I} \arrow[bend left=60,swap]{rrrr}{\phi}[swap]{\text{Lem 4.13}} \arrow[rightarrow]{rr}{\text{Lem 4.8}}[swap]{\phi_{1}}
&
  & \tilde{\mathcal{P}}\arrow[rightarrow]{rr}{\text{Lem 4.12}}[swap]{\phi_{2}}\arrow[hookrightarrow]{dd}[swap]{\text{Lem 4.9 \& 4.10}} 
  &
  & \tilde{\mathcal{Y}}\arrow[hookrightarrow]{dd}{\text{Lem 4.11}} 
  &  \\ 
  \\
  &
  & \mathcal{P} \arrow[rightarrow]{rr}{\text{Prop 4.6}}[swap]{\psi}
  &
  & \mathcal{Y}
\end{tikzcd}

\end{center}

Finally, we will show that, for each $I \in \mathcal{I}$, the parabolic subgroup $G_{I}$ is isomorphic to the Young subgroup $Y(\rho_{I})$ and use this to prove our first main result. 

\begin{lemma}\label{latticeeqiv} There exists an order-isomorphism.
\begin{align*} &\phi_{1}: \mathcal{I}\longrightarrow \tilde{\mathcal{P}}, \\ &\phi_{1}: I\longmapsto \rho_{I}.
\end{align*} 
\end{lemma}
\begin{proof} To show that $\phi_{1}$ is an order isomorphism, we must show that for any $I, J \in \mathcal{I}$,  $I \subseteq J$ if and only if $\rho_{I} \leq \rho_{J}$.

Suppose $I \subseteq J$. Then $\Gamma_{I}$ is a subgraph of $\Gamma_{J}$, meaning the vertex set of each connected component of $\Gamma_{I}$ is a subset of the vertex set of some component of $\Gamma_{J}$ and so $\rho_{I} \leq \rho_{J}$. 

If $\rho_{I} \leq \rho_{J}$ then each part of $\rho_{I}$ is a subset of a part of $\rho_{J}$, meaning the vertex set of each connected component of $\Gamma_{I}$ is a subset of the vertex set of some component of $\Gamma_{J}$. Recall $\Gamma_{I}$ is obtained by deleting all edges of $\Gamma_{Q}$ corresponding to all  $t_{i} \notin I$. If there existed some $t_{i} \in I \setminus J$, then $E_{i}$ would lie in a connected component of $\Gamma_{I}$, so the vertex set of this component would contain the endpoints of $E_{i}$. However, as $t_{i} \notin J$, $E_{i}$ would not lie in any connected component of $\Gamma_{J}$. Thus the endpoints of $E_{i}$ would lie in separate connected components of $\Gamma_{J}$, contradicting the fact that the vertex set of each connected component of $\Gamma_{I}$ is a subset of the vertex set of some component of $\Gamma_{J}$. Thus $I \subseteq J$. 

\end{proof}

We go on to show that $\tilde{\mathcal{P}}$ is a lattice and so $\phi_{1}$ is a lattice isomorphism. 

\begin{lemma}\label{phi1} The set $\tilde{\mathcal{P}}$ is a lattice under the refinement ordering. 
\end{lemma}
\begin{proof} We prove that $\tilde{\mathcal{P}}$ is a lattice under refinement by showing that it is a sublattice of $\mathcal{P}$. 

By definition, to show that $\tilde{\mathcal{P}}$ is a sublattice, we must show that for any $I, J \in \mathcal{I}$,  $\rho_{I}\vee \rho_{J}, \rho_{I}\wedge \rho_{J} \in \tilde{\mathcal{P}} $. We show that 
\begin{align}
    \rho_{I\cap J} &= \rho_{I}\wedge \rho_{J}, \\
\rho_{I\cup J} &= \rho_{I}\vee \rho_{J}.
\end{align}

For ease of notation we will write $R_{K}$ for the equivalence relation corresponding to the partition $\rho_{K} \in \tilde{\mathcal{P}}$. Noting Example $\ref{latticeeg}$ $(2.)$, we need to show that $\rho_{I\cap J}$ is precisely the partition corresponding to the equivalence relation $R_{I}\cap R_{J}$ and $\rho_{I\cup J}$ is precisely the partition corresponding to the equivalence relation $t(R_{I}\cup R_{J})$.

Let $\rho_{I} = \sqcup_{i=1}^{a}\alpha_{i}$, $\rho_{J}= \sqcup_{j=1}^{b}\alpha'_{j}$ and $\rho_{I\cap J} = \sqcup_{l=1}^{c}\beta_{l}$. First, we show that for all $x, y \in \lbrace 1,..., n+r \rbrace$, $$xR_{I\cap J}y \Leftrightarrow x(R_{I}\cap R_{J})y.$$

Recall that for any $K \in \mathcal{I}$, $\rho_{K}$ is obtained by deleting the edges in $\Gamma_{Q}$ corresponding to the $t_{i}$ not in $K$ then taking the parts of $\rho_{K}$ to be the vertex sets of the connected components of this resulting graph, denoted by $\Gamma_{K}$.  

As $I\cap J \subseteq I, J$ and as $\Gamma_{Q}$ is a tree, for any two distinct vertices $x$ and $y$ lying in the same connected component of $\Gamma_{I\cap J }$, the unique path from $x$ to $y$ in $\Gamma_{Q}$ must consist only of edges corresponding to some $t_{i} \in I\cap J$. Therefore there exists a path from $x$ to $y$ in both $\Gamma_{I}$ and $\Gamma_{J}$. Thus a connected component in $\Gamma_{I\cap J }$ is a subgraph of some connected component in both $\Gamma_{I}$ and in $\Gamma_{J}$. So $xR_{I\cap J}y$ implies $ x(R_{I}\cap R_{J})y.$

Conversely, suppose two distinct vertices $x$ and $y$ lie in the same component in $\Gamma_{I}$ and the same component in $\Gamma_{J}$. Then there exists a path in $\Gamma_{Q}$ between $x$ and $y$ consisting only of edges corresponding to some $t_{i} \in I$ and a path between $x$ and $y$ consisting only of edges corresponding to some $t_{i} \in J$. However, $\Gamma_{Q}$ is a tree, meaning any existing path between $x$ and $y$ is unique. Thus the edges in the path between $x$ and $y$ in $\Gamma_{Q}$ consist only of edges corresponding to some $t_{i} \in I\cap J$, giving that $x$ and $y$ lie in the same connected component of $\Gamma_{I\cap J}$. So $ x(R_{I}\cap R_{J})y$ implies $xR_{I\cap J}y$. 

As the equivalence relation corresponding to $\rho_{I\cap J}$ is precisely $R_{I}\cap R_{J}$, it must be that $\rho_{I\cap J} = \rho_{I}\wedge \rho_{J}$, so $(1)$ is shown. 

Finally, we show that for all $x, y \in \lbrace 1,..., n+r \rbrace$, $$xR_{I\cup J}y \Leftrightarrow x(t(R_{I}\cup R_{J}))y.$$

For any $x, y \in \lbrace 1,..., n+r \rbrace, xR_{I\cup J}y$ if and only if $x$ and $y$ lie in the same connected component of $\Gamma_{I\cup J}$. This occurs if and only if there exists a unique path in $\Gamma_{Q}$ between $x$ and $y$ consisting only of edges corresponding to elements of $ I\cup J$. That is, if and only if there exists distinct $z_{0},..., z_{m} \in \lbrace 1,..., n+r \rbrace$ such that $ x=z_{0}, y=z_{m}$ and for all $1\leq i \leq m$ either $t_{i} \in I$ or $t_{i} \in J$, where $t_{i}$ corresponds to the edge in $\Gamma_{Q}$ with endpoints $(z_{i-1}, z_{i})$. In other words, there exist distinct $z_{0},..., z_{m} \in \lbrace 1,..., n+r \rbrace$ such that $ x=z_{0}, y=z_{m}$ and for all $1\leq i \leq m-1$ either $ z_{i}R_{1}z_{i+1} $ or $ z_{i}R_{2}z_{i+1}.$ So $xR_{I\cup J}y \Leftrightarrow x(t(R_{I}\cup R_{J}))y.$

Hence both $(1)$ and $(2)$ are shown, so $\rho_{I}\vee \rho_{J}, \rho_{I}\wedge \rho_{J} \in \tilde{\mathcal{P}}$ for all $I, J \in \mathcal{I}$ and so $\tilde{\mathcal{P}}$ is a sublattice of $\mathcal{P}$.
\end{proof}

From Lemma \ref{latticeeqiv} together with Lemma \ref{phi1} and Proposition \ref{orderpreserving} we conclude the following.

\begin{corollary} The set $\tilde{\mathcal{P}}$ is a sublattice of $\mathcal{P}$ isomorphic to $\mathcal{I}$.
\end{corollary} We now show that $\tilde{\mathcal{Y}}$ is a sublattice of $\mathcal{Y}$.
\begin{lemma}\label{sublattice} The set $\tilde{\mathcal{Y}}$ is a sublattice of $\mathcal{Y}$.
\end{lemma}
\begin{proof} For all $I, J \in \mathcal{I}$, 
\begin{align*}
Y(\rho_{I})\vee Y(\rho_{J}) &= \psi(\rho_{I})\vee\psi(\rho_{J}) \\
&= \psi(\rho_{I}\vee\rho_{J}) \text{ by Proposition \ref{psi} }\\
&= \psi(\rho_{I\cup J}) \text{ by Lemma \ref{phi1} } \\
&=Y(\rho_{I\cup J}) \in \tilde{\mathcal{Y}}
\end{align*} and 
\begin{align*}
Y(\rho_{I})\wedge Y(\rho_{J}) &= \psi(\rho_{I})\wedge\psi(\rho_{J}) \\
&= \psi(\rho_{I}\wedge\rho_{J}) \text{ by Proposition \ref{psi} }\\
&= \psi(\rho_{I\cap J}) \text{ by Lemma \ref{phi1} } \\
&=Y(\rho_{I\cap J}) \in \tilde{\mathcal{Y}}.
\end{align*} Thus $\tilde{\mathcal{Y}}$ is a sublattice of $\mathcal{Y}$. 
\end{proof}
\begin{lemma}\label{phi2} There exists a lattice isomorphism 
\begin{align*}
&\phi_{2}: \tilde{\mathcal{P}}\longrightarrow \tilde{\mathcal{Y}}, \\ &\phi_{2}: \rho_{I} \longmapsto Y(\rho_{I}).
\end{align*}
\end{lemma} 
\begin{proof} It follows from Proposition $\ref{psi}$ that 
\begin{align*}
&\psi\vert_{\tilde{\mathcal{P}}}: \tilde{\mathcal{P}} \longrightarrow \mathcal{Y}, \\ &\psi\vert_{\tilde{\mathcal{P}}}: \rho_{I} \longmapsto Y(\rho_{I})
\end{align*} is an injective lattice homomorphism. As $\tilde{\mathcal{Y}} = im(\psi\vert_{\tilde{\mathcal{P}}})$, we conclude that $\psi\vert_{\tilde{\mathcal{P}}}$ induces the lattice isomorphism:
\begin{align*}
&\phi_{2}: \tilde{\mathcal{P}}\longrightarrow \tilde{\mathcal{Y}}, \\ &\phi_{2}: \rho_{I} \longmapsto Y(\rho_{I}).
\end{align*}

\end{proof}

Thus $\tilde{\mathcal{Y}}$ is a sublattice of $\mathcal{Y}$ isomorphic to $\tilde{\mathcal{P}}$ and so isomorphic to $\mathcal{I}$. 
\begin{lemma}\label{phi7} There exists a lattice isomorphism 
\begin{align*}
&\phi: \mathcal{I}\longrightarrow \tilde{\mathcal{Y}}, \\ &\phi: I\longmapsto Y(\rho_{I}).
\end{align*}
\end{lemma}
\begin{proof} Let $$\phi: \mathcal{I}\longrightarrow \tilde{\mathcal{Y}}$$ be the composition $\phi = \phi_{2} \circ \phi_{1}$. Then $$\phi(I) = Y(\rho_{I})$$ for all $I \in \mathcal{I}$. By Lemma $\ref{phi1}$ and Lemma $\ref{phi2}$, $\phi$ is a lattice isomorphism.  
\end{proof}
Recall that $\mathcal{G}$ is the set of subgroups of $G_{Q}$ (so $\mathcal{G}$ is a lattice under inclusion) and $\tilde{\mathcal{G}} = \lbrace G_{I} : I \in \mathcal{I} \rbrace$ (the set of parabolic subgroups of $G_{Q}$). Our next aim is to show that $\tilde{\mathcal{G}} $ is a sublattice of $\mathcal{G}$ . For this we need the following lemmas. 
\begin{lemma}\label{G} For $I \in \mathcal{I}$, the map \begin{align*}
\pi_{Q}\mid_{G_{I}}&: G_{I} \longrightarrow Y(\rho_{I}), \\
\pi_{Q}\mid_{G_{I} }&: t_{i} \longmapsto (x_{i}, y_{i})
\end{align*} is a group isomorphism. That is, $G_{I} \cong Y(\rho_{I})$. 
\end{lemma}
\begin{proof} For ease of notation, we write $\pi\mid_{I}$ in place of $\pi_{Q}\mid_{G_{I} }$.
 
Clearly, $im( \pi\mid_{I}) \subseteq Y(\rho_{I})$ as for each $t_{i} \in I$, $E_{i} = \lbrace x_{i}, y_{i}\rbrace$ lies in some connected component $\Gamma^{j}_{I}$ of $\Gamma_{I}$. It follows that $x_{i}, y_{i} \in \alpha_{j}$ and so $\pi\mid_{I}(t_{i}) = (x_{i}, y_{i}) \in \Sigma_{\alpha_{j}} \subseteq Y(\rho_{I})$. 

Since $\pi_{Q}$ is an isomorphism, $\pi\mid_{I}: G_{I} \longmapsto Y(\rho_{I})$ is injective. Hence it is enough to show that $im(\pi \mid_{I}) = Y(\rho_{I})$ in order to show that $\pi\mid_{I}: G_{I} \longrightarrow Y(\rho_{I})$ is an isomorphism. 

Let $\rho_{I} = \sqcup_{j=1}^{k}\alpha_{j}$. For any $1 \leq j \leq k$, $\alpha_{j} = \lbrace m_{1},..., m_{u} \rbrace \subseteq \lbrace 1,..., n+r \rbrace$, we have $\Sigma_{\alpha_{j}} \cong \Sigma_{u}$ induced by the map
$$\tilde{\pi}: l \longmapsto m_{l}$$

As the set of elementary transpositions $\lbrace (l, l+1): 1 \leq j \leq u-1 \rbrace$ generates $\Sigma_{u}$, the set $\lbrace (m_{l}, m_{l+1}): 1 \leq j \leq u-1 \rbrace$ generates $\Sigma_{\alpha_{j}}$. For any $1 \leq l \leq u-1$, there exists a unique path in $\Gamma^{j}_{I}$ between the vertices $m_{l}$ and $m_{l+1}$ (as $m_{l}, m_{l+1} \in \alpha_{j}$):

\xymatrix{&&m_{l}=x_{1}\ar@{-}^{E_{i_{1}}}[r]&x_{2}\ar@{-}^{E_{i_{2}}}[r]&x_{3}\ar@{-}^{E_{i_{3}}}[r]&...\ar@{-}^{E_{i_{p-1}}}[r]&x_{p}=m_{l+1}}
and 
\begin{align*}
(m_{l}, m_{l+1}) &= (x_{1}, x_{2})(x_{2}, x_{3})...(x_{p-2}, x_{p-1})(x_{p-1}, x_{p})(x_{p-2}, x_{p-1})\\
&...(x_{2}, x_{3})(x_{1}, x_{2}) \\
&= \pi\mid_{I}(t_{i_{1}}t_{i_{2}}...t_{i_{p-1}}t_{i_{p}}t_{i_{p-1}}...t_{i_{2}}t_{i_{1}})
\end{align*} 
where $t_{i_{q}} \in I$ for each $1 \leq q \leq p$, as each edge $E_{i_{q}}$ lay in $\Gamma^{j}_{I}$. Thus $\pi\mid_{I}$ is surjective and so $G_{I} \cong Y(\rho_{I})$. 
\end{proof} 

\begin{theorem}\label{latticeiso}
The subset $\tilde{\mathcal{G}}$ is a sublattice of $\mathcal{G}$ and there exists a lattice isomorphism $$\Phi: \mathcal{I}\longrightarrow \tilde{\mathcal{G}},$$
$$\Phi: I\longmapsto G_{I}.$$
\end{theorem}
\begin{proof} Let $Sub(\Sigma_{n+r})$ denote the set of subgroups of $\Sigma_{n+r}$. The group isomorphism $$(\pi_{Q})^{-1}: \Sigma_{n+r} \longrightarrow G_{Q}$$ induces a lattice isomorphism 
\begin{align*} Sub(\Sigma_{n+r})& \longrightarrow \mathcal{G}, \\
H& \longmapsto (\pi_{Q})^{-1}(H).
\end{align*}
Note that, by Lemma $\ref{G}$, $(\pi_{Q})^{-1}(Y(\rho_{I})) = G_{I}$ for each $Y(\rho_{I}) \in \tilde{\mathcal{Y}}$. Thus $\tilde{\mathcal{G}}$ is the image of the sublattice  $\tilde{\mathcal{Y}}$ under this induced lattice isomorphism. It follows that $\tilde{\mathcal{G}}$ is a sublattice of $\mathcal{G}$ and that taking the restriction of the induced lattice isomorphism to the sublattice $\tilde{\mathcal{Y}}$ of $Sub(\Sigma_{n+r})$ yields the following lattice isomorphism 
\begin{align*} \tilde{\mathcal{Y}} &\longrightarrow \tilde{\mathcal{G}}, \\
Y(\rho_{I}) &\longmapsto G_{I}.
\end{align*} Composing this map with $\phi$ (Lemma \ref{phi7}) will give the desired lattice isomorphism $\Phi$.
\end{proof}

Lemma $\ref{G}$ allows us to prove a further result for cluster groups of mutation-Dynkin type $A_{n_{1}}\sqcup...\sqcup A_{n_{r}}$. 

\begin{proposition}\label{parabolicgenerators} For any $ I \in \mathcal{I}$, $G_{I}\cap \lbrace t_{1},.., t_{n} \rbrace = I$.
\end{proposition}
\begin{proof} Suppose $Q$ is of mutation-Dynkin type $A_{n_{1}}\sqcup...\sqcup A_{n_{r}}$ with $ \Sigma_{i=1}^{r}n_{i} = n$ for some $r \geq 1$. Fix any $t_{q} \notin I$ and let $J = \lbrace t_{1},.., t_{n}\rbrace \setminus \lbrace t_{q} \rbrace$.

So $\Gamma_{Q}$ has connected components $\Gamma_{Q_{1}}, ..., \Gamma_{Q_{r}}$ where $\Gamma_{Q_{i}}$ is the braid graph corresponding to the indecomposable component $Q_{i}$. Moreover, $E_{q} = \lbrace x_{q}, y_{q} \rbrace$ is the edge in $\Gamma_{Q}$ corresponding to $t_{q}$ for some distinct $x_{q}, y_{q} \in \lbrace 1,..., n+r\rbrace$. Thus $\Gamma_{J}$ will consist of $r+1$ connected components $$\Gamma_{Q_{1}}, ..., \Gamma_{Q_{m-1}}, \Gamma_{Q_{m}}^{1}, \Gamma_{Q_{m}}^{2}, \Gamma_{Q_{m+1}},..., \Gamma_{Q_{r}}$$ where $\Gamma_{Q_{m}}$ is the component containing $E_{q}$ and $\Gamma_{Q_{m}}^{1}$ and $ \Gamma_{Q_{m}}^{2}$ are the two connected components of the graph obtained from $\Gamma_{Q_{m}}$ by deleting the edge $E_{q}$. 

Without loss of generality, we can assume $x_{q} \in \alpha_{1}$ and $y_{q} \in \alpha_{2}$ where $\alpha_{1}$ is the vertex set of $\Gamma_{Q_{m}}^{1}$ and $\alpha_{2}$ is the vertex set of $\Gamma_{Q_{m}}^{2}$

Denote the vertex set of each $\Gamma_{Q_{j}}$ by $\alpha_{q_{j}}$. It follows that $\rho_{J} = (\sqcup_{j=1}^{m-1}\alpha_{q_{j}})\sqcup(\alpha_{1}\sqcup \alpha_{2})\sqcup(\sqcup_{j=m+1}^{r}\alpha_{q_{j}})$ where $x_{q} \in \alpha_{1}$ and $y_{q} \in \alpha_{2}$ and $$Y(\rho_{J}) = \Sigma_{\alpha_{q_{1}}} \times ...\times \Sigma_{\alpha_{q_{m-1}}} \times  \Sigma_{\alpha_{1}} \times \Sigma_{\alpha_{2}} \times \Sigma_{\alpha_{q_{m+1}}} \times ...\times \Sigma_{\alpha_{q_{r}}}.$$ 

Note that for any $\sigma \in Y(\rho_{J})$ and $m \in \alpha_{1}$, $\sigma(m) \in \alpha_{1}$. Thus $\sigma(x_{q}) \neq y_{q}$ for any $\sigma \in Y(\rho_{J})$ and so $\pi_{Q}(t_{q}) = (x_{q}, y_{q}) \notin Y(\rho_{J})$.

As $I \subseteq J$, we have $Y(\rho_{I}) \subseteq Y(\rho_{J})$ (by Lemma $\ref{latticeeqiv}$ and Proposition $\ref{psi}$) so $\pi(t_{q}) \notin Y(\rho_{I})$ hence $\pi_{Q}(t_{q}) \notin im(\pi\mid_{I})$, so $t_{q} \notin G_{I}$ by Lemma $\ref{G}$. 

Thus $G_{I}\cap \lbrace t_{1},.., t_{n} \rbrace\subseteq I$. The reverse inclusion is obvious as $I \subseteq G_{I}$ and $I \subseteq \lbrace t_{1},.., t_{n} \rbrace$.

Note that in the case when $r=1$ (so $Q$ is of mutation-Dynkin type $A_{n}$) $\Gamma_{Q}$ will contain just one connected component. Thus $\Gamma_{Q_{m}}$ will simply be $\Gamma_{Q}$ meaning $\Gamma_{J}$ will consist of two connected components $\Gamma_{Q}^{1}$ and $\Gamma_{Q}^{2}$ and the proof for this case continues as for when $r >1$. 

\end{proof}
\section{parabolic subgroups are cluster groups}

Suppose $Q$ is a quiver of mutation-Dynkin type with connected components $Q_{1},..., Q_{r}$ of type $A_{n_{1}},..., A_{n_{r}}$, respectively, where $n_{1},..., n_{r} \in \mathbb{Z}^{+}$ and $ n = \Sigma_{i=1}^{r} n_{i}$. In this section, we prove that parabolic subgroups of the cluster group $G_{Q}$ have presentations given by restricting the cluster group presentation of the whole group. 

\begin{proposition}\label{subquiver}
\begin{itemize}
\item[(A)]\cite{FWZ} Any full subquiver of an indecomposable quiver of mutation-Dynkin type is a disjoint union  of quivers of mutation-Dynkin type. 
\item[(B)] Any full subquiver of a mutation-Dynkin type $A_{n}$ quiver is of mutation-Dynkin type $A_{n_{1}}\sqcup...\sqcup A_{n_{r}}$ for some $r \geq 1$.
\end{itemize} 
\end{proposition}
\begin{proof} For a proof of the statement $(A)$, see \cite{FWZ}.

The second statement $(B)$ follows from Theorem $\ref{triangulationquivers}$ as taking a full subquiver of a mutation-Dynkin type $A_{n}$ quiver is equivalent to cutting along a diagonal of the corresponding triangulation of an $(n+3)$-gon to obtain a disjoint union of triangulations of smaller polygons.

\end{proof}

Consider a subset $I \subseteq T$. Let $Q_{I}$ denote the full subquiver, $Q_{I}$, of $Q$ on vertices corresponding to the elements of $I$. Consider the group with presentation given by taking the defining generators to be the set $I$ and the set of relations to be all the corresponding defining relations of the cluster presentation of $G_{Q}$ that consist only of elements of $I$. The group corresponding to this presentation is a cluster group of mutation-Dynkin type as it is precisely $G_{Q_{I}}$. By Proposition \ref{subquiver}, $Q_{I}$ will be a disjoint union of quivers $Q'_{1},..., Q'_{r'}$ of mutation-Dynkin type $A_{n'_{1}},..., A_{n'_{r'}}$, respectively, for some $1 \leq n'_{i} \leq n$ and $r' \geq 1$.

Consider the parabolic subgroup, $G_{I}$, generated by $I$. From Lemma $\ref{G}$, $G_{I}$ is isomorphic to $Y(\rho_{I})$. 
We will show that there exists an isomorphism between $G_{Q_{I}}$ and $Y(\rho_{I})$ and so, by transitivity, the parabolic subgroup $G_{I}$ has a cluster group presentation associated to $Q_{I}$.

In particular, we will show that this isomorphism gives the following commutative diagram.\hspace{4cm} \xymatrix{\ar @{} [dr] |{\circlearrowleft} G_{I} \ar[d] \ar[r] & G_{Q} \\Y(\rho_{I}) \ar[r]& G_{Q_{I}} \ar[u]}

To distinguish between the defining generators of $G_{Q}$ and $G_{Q_{I}}$, we label the generators of the cluster presentation of $G_{Q_{I}}$ by $t'_{i}$, for each $t_{i} \in I$.

\begin{theorem}\label{paraboliccluster}
Suppose $Q$ is a quiver of mutation-Dynkin type with connected components $Q_{1},..., Q_{r}$ of type $A_{n_{1}},..., A_{n_{r}}$ for some $n_{1},..., n_{r} \in \mathbb{Z}^{+}$. For any $I \subset T$, there exists an isomorphism 
\begin{align*}
G_{Q_{I}} &\longrightarrow G_{I}, \\
t'_{i} &\longmapsto t_{i}.
\end{align*}
\end{theorem}
\begin{proof} We show that there exists an isomorphism
\begin{align*}
G_{Q_{I}} &\longrightarrow Y(\rho_{I}), \\
t'_{i} &\longmapsto \pi_{I}(t_{i}).
\end{align*}

We begin by considering the braid graphs $\Gamma_{Q}$ and $\Gamma_{Q_{I}}$ and show that  $ \Gamma_{I} = \Gamma_{Q_{I}} $.

Let $\mathcal{T}$ be a triangulation giving rise to $Q$. From the proof of Proposition $\ref{subquiver}$, a triangulation $\mathcal{T'}$ giving rise to $Q_{I}$ can be obtained from $\mathcal{T}$ by cutting along each of the diagonals lying in the set $ \lbrace d_{i} : t_{i} \notin I \rbrace$. 

We consider the braid graph, $\Gamma_{Q_{I}}$ of $Q_{I}$. We remark that in obtaining $\mathcal{T'}$, no triangle of $\mathcal{T}$ will have been deleted. Thus the number of vertices of $\Gamma_{Q_{I}}$ equals the number of vertices in $\Gamma_{Q}$, which equals the number of vertices in $\Gamma_{I}$. Moreover, all triangles in $\mathcal{T'}$ will have the same orientation as in $\mathcal{T}$ except those which were bounded by a diagonal that was `cut'. As one diagonal of $\mathcal{T}$ bounds exactly two triangles, the set of edges in $\Gamma_{Q_{I}}$ will equal the set of edges in $\Gamma_{Q}$ minus the set of edges which correspond to one of the `cut' diagonals. These are precisely the edges corresponding to all elements of $T\setminus I$. Thus $\Gamma_{Q_{I}}$ is equal to the graph obtained from $\Gamma_{Q}$ by deleting  the edges corresponding to the $t_{i} \notin I$, which is precisely $\Gamma_{I}$.

The fixed labelling of $\Gamma_{Q}$ induces a labelling on the vertices $\Gamma_{Q_{I}}$ from which we obtain a partition $\rho'$ of $\lbrace 1,..., n+r\rbrace$. As $\Gamma_{Q_{I}} = \Gamma_{I}$, this partition is precisely $\rho_{I}$ meaning $Y(\rho') = Y(\rho_{I})$. 

By Lemma $\ref{braidgraphiso2}$, there exists an isomorphism
\begin{align*}
\pi_{Q_{I}}:G_{Q_{I}}&\longrightarrow Y(\rho'), \\
\pi_{Q_{I}}:t'_{i}&\longmapsto (x_{i}, y_{i})
\end{align*} where $x_{i}$ and $y_{i}$ are the endpoints of the edge in $\Gamma_{Q_{I}}$ corresponding to $t'_{i}$. Moreover, as $\Gamma_{Q_{I}} = \Gamma_{I}$ and applying by Lemma $\ref{G}$, the following map is an isomorphism. 
\begin{align*}
\pi_{Q}\mid_{I}:G_{I}&\longrightarrow Y(\rho_{I}), \\
\pi_{Q}\mid_{I}:t_{i}&\longmapsto (x_{i}, y_{i}).
\end{align*}
As $Y(\rho') = Y(\rho_{I})$, we can define $\pi := (\pi_{Q}\mid_{I})^{-1} \circ \pi_{Q_{I}}$, which will be the desired isomorphism,
\begin{align*}
\pi_{Q}\mid_{I}^{-1} \circ \pi_{Q_{I}}: G_{Q'} &\longrightarrow G_{I}, \\
\pi_{Q}\mid_{I}^{-1} \circ \pi_{Q_{I}}: t'_{i} &\longrightarrow t_{i}.
\end{align*}
\end{proof}

 \end{document}